\documentclass[leqno, final]{article}
\usepackage{ifdraft}
\ifdraft{\usepackage[a5paper, margin=5mm]{geometry}}{}
\usepackage[disable=ifdraft]{microtype}
\usepackage{amsthm, amsmath, amssymb, mathtools}
\usepackage{enumitem}

\usepackage[maxbibnames=99, backend=biber,
	style=numeric, doi=false, url=true, isbn=false, eprint=false]{biblatex}
\newtheorem{corollary}{Corollary}
\newtheorem*{corollary*}{Corollary}
\newtheorem{lemma}{Lemma}
\newtheorem*{lemma*}{Lemma}
\newtheorem{proposition}{Proposition}
\newtheorem*{theorem*}{Theorem} 
\newtheorem*{theorem1}{Theorem 1} 
\newtheorem*{theorem2}{Theorem 2} 
\newtheorem{theorem}{Theorem}
\theoremstyle{definition}
\newtheorem{definition}{Definition}
\newtheorem{example}{Example}
\newtheorem{question}{Question}
\newtheorem{remark}{Remark}

\usepackage[capitalize]{cleveref}
\crefname{lemma}{Lemma}{Lemma}
\crefname{proposition}{Proposition}{Proposition}
\crefname{question}{Question}{Question}
\crefname{remark}{Remark}{Remark}
\crefname{theorem}{Theorem}{Theorem}

\newcommand{\xto}[1]{\stackrel{#1}{\to}}

\newcommand{\reals}{\mathbb{R}}
\newcommand{\ball}{\mathbb{B}}
\newcommand{\disk}{\mathbb{D}}
\newcommand{\plane}{\mathbb{C}}
\newcommand{\dirichlet}{\mathcal{D}}
\renewcommand{\natural}{\mathbb{N}}
\newcommand{\torus}{\mathbb{T}}
\newcommand{\sphere}{\mathbb{S}}
\newcommand{\Hilbert}{\mathcal{H}}

\newcommand{\rand}{\partial}
\renewcommand{\phi}{\varphi}
\renewcommand{\bar}{\overline}
\newcommand{\la}{\langle}
\newcommand{\ra}{\rangle}

\DeclareMathOperator{\Fact}{Fact}
\DeclareMathOperator{\Mult}{Mult}
\DeclareMathOperator{\Comp}{Comp}

\DeclareMathOperator{\Int}{Int}

\title{Boundary values via reproducing kernels: \\
	The Julia-Carathéodory theorem}
\author{Frej Dahlin}

\begin{document}
\maketitle
\begin{abstract}
Given a reproducing kernel $k$ on a nonempty set $X$, we define the reproductive
boundary of $X$ with respect to $k$. Furthermore, we generalize the well known
nontangential and horocyclic approach regions of the unit circle to this new
kind of boundary. We also introduce the concept of a composition factor of
$k$, an abstract analogue of analytic selfmaps of the unit disk. Using these
notions, we obtain a far reaching generalization of the Julia-Carathéodory
theorem, stated on an arbitrary set. We also prove Julia's lemma in the abstract
setting and give sufficient conditions for the convergence of iterates of some
selfmaps. As an application we improve the classical theorem on the unit disk
for contractive multipliers of standard weighted Dirichlet spaces, as well as
Besov spaces on the unit ball. Many examples and questions are provided for
these novel objects of study.
\end{abstract}

\tableofcontents

\section{Introduction}
In this paper $\disk$ denotes the unit disk in the complex plane $\plane$ and
$\torus$ its boundary, the unit circle. We recall the classical Julia-Carathéodory
theorem.
\begin{theorem*}[Julia-Carathéodory]

Given a point
$\zeta\in\torus$ and an analytic selfmap of the unit disk $\phi:\disk\to\disk$,
then the following are equivalent:
\begin{enumerate}[label=(\roman*)]
\item \[
	\liminf_{z\to\zeta}\frac{1-|\phi(z)|^2}{1-|z|^2} = c < \infty.
\]
\item There is a $\lambda\in\torus$ such that the difference quotient \[
	\frac{1-\phi(z)\overline\lambda}{1-z\overline\zeta}
\]
has nontangential limit $c$ at $\zeta$.
\item $\phi$ and its derivative $\phi'$ has nontangential limits equal to
$\lambda\in\torus$ and $c\lambda\overline\zeta$ at $\zeta$, respectively.
\end{enumerate}
Furthermore, if any of the above hold, then $c>0$ and $\phi$ has horocyclic limit $\lambda$ at $\zeta$, which is the same in (ii)
and (iii), in the sense that $\phi(z_n)\to\lambda$ along any sequence $z_n\to\zeta$
constrained to a fixed set $E(M,\zeta) = \{z\in\disk:|\zeta-z|^2\leq M(1-|z|^2)\},\,M>0$.
\end{theorem*}

The condition (i) was considered by Julia \cite{julia} in 1920, and his contribution
surrounding it is often referred to as \emph{Julia's lemma}. Specifically, he proved that (i)
implies the existence of $\lambda\in\torus$ such that
\begin{equation}\label{eq:julias-lemma}
	\frac{|1-\phi(z)\bar\lambda|^2}{1-|\phi(z)|^2}
	\leq c\frac{|1-z\bar\zeta|^2}{1-|z|^2}.
\end{equation}
In other words, $\phi(E(M,\zeta))\subset E(cM,\lambda)$, which proves that $\phi$ has
horocyclic limit $\lambda$ at $\zeta$. Moreover, if equality holds in
\eqref{eq:julias-lemma} for some $z\in\disk$, then $\phi$ is an automorphism.
\\ In 1929 Carathéodory \cite{caratheodory} completed the above equivalences
showing that (iii), $\phi$ has an angular derivative. This classical
object is still being studied today, see for example the recent note
\cite{bergman_angular}.

Generalizations of the Julia-Carathéodory theorem typically extend it to different
domains in several variables. For example, the unit ball $\ball_d$ in $\plane^d$
by Rudin \cite{rudin_unit-ball}, and the polydisk $\disk^d$ initiated by Jafari,
but completed by Abate \cite{jafari_polydisk, abate_polydisk}.
One should also mention the recent work in the bidisk by Agler, McCarthy, and Young, as
well as Jury and Tsikalas
\cite{agler-mccarthy_bidisk, jury-tsikalas_denjoy-wolff}, which exploit the special structure
of this domain. Those methods yield results of considerable depth, but do not seem to
extend to higher dimensions.

The results mentioned have all been obtained using geometric considerations.
However, Sarason discovered a proof of the classical Julia-Carathéodory
theorem using the theory of Hilbert spaces that have a reproducing kernel
\cite{sarason_julia-caratheodory, sarason}.

These are Hilbert spaces of
functions on some nonempty set $X$, $f:X\to\plane$ such that for every $x\in
X$ the point evaluation functional $f\mapsto f(x)$ is bounded. By the Riesz
representation theorem there corresponds to every $x\in X$ a \emph{kernel function} $k_x$ such
that $f(x) = \la f,k_x\ra$. The two variable function
$k(x,y) = \la k_y, k_x \ra $ is called the \emph{reproducing kernel} and we denote
its corresponding Hilbert space by $\Hilbert(k)$.
All of the required theory of reproducing kernels is recalled in \S2.1.

We shall state Sarason's version of the Julia-Carath\'eodory theorem,
but first note that for the Szeg\H{o} kernel $k(z,w)=(1 - z\bar w)^{-1}$ on
$\disk$ and any analytic selfmap
$\phi:\disk\to\disk$, then \[
	k^\phi(z,w) = \frac{1 - \phi(z)\bar{\phi(w)}}{1 - z\bar w}
\]
defines a reproducing kernel on $\disk$ and its corresponding Hilbert space $\Hilbert(k^\phi)$,
are the well known de Branges-Rovnyak spaces.
\begin{theorem*}[Sarason]
Given a point $\zeta\in\torus$ and an analytic selfmap $\phi$ of $\disk$
$\phi:\disk\to\disk$, then the following are equivalent:
\begin{enumerate}[label=(\roman*)]
\item \[
	\liminf_{z\to\zeta}\frac{1-|\phi(z)|^2}{1-|z|^2} = c < \infty,
\]
\item There is a $\lambda\in\torus$ such that the function defined by\[
	q_\zeta(z) = \frac{1-\phi(z)\overline\lambda}{1-z\overline\zeta}
\]
belongs to $\Hilbert(k^\phi)$.
\item All functions in $\Hilbert(k^\phi)$ have nontangential limits at $\zeta$.
\end{enumerate}
Furthermore, if any of the above hold, then $\phi$ has horocyclic limit $\lambda$ at $\zeta$ and
$\|q_\zeta\|^2=c$.
\end{theorem*}
\noindent This theorem does not directly include the notion of angular derivatives. However,
one can prove using (ii) and (iii) combined that $\phi$ has an angular derivative at
$\zeta$, thus recovering the classical theorem.

The above was generalized to the case of the Drury-Arveson kernel on the
unit ball by Jury \cite{jury_julia-caratheodory}. In this paper we obtain a
far reaching generalization of the Julia-Carathéodory theorem stated only in
terms of an arbitrary  nonempty set and reproducing kernels on it. A central
difficulty that we have to address is to find abstract analogues of:  the
boundary of $\disk$, the nontangential approach regions of $\torus$, and the
analytic selfmaps of $\disk$.

Let $k$ be a reproducing kernel on a nonempty set $X$.
The intuition behind a general notion of boundary is the identification of points
$x\in X$ with the kernel functions $k_x$ together with the following simple
procedure that embeds $X$ into an `extended' set. More precisely, let $\widehat X$
be the set of sequences $\hat x = (x_n)$ of $X$ such that the sequence of kernel functions $(k_{x_n})$
converge pointwise. Two sequences $\hat x,\hat y\in\widehat X$ are equivalent,
$\hat x\sim\hat y$ if and only if the corresponding sequences of functions $(k_{x_n})$ and
$(k_{y_n})$ converge pointwise to the same function; we let $\widetilde X$ denote the
quotient set. For $\xi\in\widetilde X$ we denote the above limiting function as $k_\xi$.
For these functions the following possibilities can occur:
\begin{enumerate}[label=(\roman*)]
\item $k_\xi = k_x$ for some $x\in X$,
\item $k_\xi\in\Hilbert(k)$ but (i) fails,
\item $k_\xi\notin\Hilbert(k)$.
\end{enumerate}
We define the \emph{reproductive boundary} of $X$ with respect to $k$ to be the
set of points $\xi\in\widetilde X$ that satisfy (iii). We denote this set by $\partial_kX$,
and call their corresponding functions \emph{boundary kernel functions} or
\emph{boundary functions} for short. For a sequence $(x_n)\in \widehat X$
we will write $x_n\xto{k}\xi$ to mean that $[(x_n)]_\sim=\xi\in\widetilde X$, similarly
we write $x\xto{k}\xi$ to mean \emph{along any sequence $(x_n)$ such that $x_n\xto{k}\xi$}.
For more properties of the reproductive boundary and a discussion of (i) and (ii) see \S2.3.
\\ Now we are able to generalize the nontangential and horocyclic approach regions, respectively.
For $\xi\in\partial_kX$ and $M>0$ consider the family subsets of $X$ \begin{align}
	\Gamma_k(M,\xi) &= \{x\in X: k(x,x)\leq M|k_\xi(x)|\}, \\
	E_k(M,\xi) &= \{x\in X: k(x,x)\leq M|k_\xi(x)|^2\}.
\end{align}
A set $\Gamma_k(M,\xi)$ \emph{approaches} $\xi$ if there is a sequence $x_n\xto{k}\xi$
within this set, if along any such sequence the function $f:X\to\plane$ has a limit, then
we say that $f$ has a $\Gamma_k$-limit at $\xi$. Of course similar definitions hold for
the $E_k$ family of sets, see \S2.4 for the complete definitions.
\\ Supposing that $k$ is nonzero, we shall generalize the notion of
analytic selfmaps of $\disk$ as the selfmaps $\phi:X\to X$ such that the function
\begin{equation}\label{eq:composition-factor}
	x,y\mapsto \frac{k(x,y)}{k(\phi(x),\phi(y))}, \qquad x,y\in X
\end{equation}
is a reproducing kernel, in this case $\phi$ is a called a \emph{composition factor} of $k$.
In general one can even allow for distinct kernels in the numerator and denominator,
see \S2.2.
\\ To illustrate these three object we give a list of many concrete examples in \S2.5.

In order to state a variant of our main theorem we need to introduce the following (mild)
regularity assumptions on $k$. Details are deferred to \S3.
\begin{enumerate}[label=(\Alph*)]
\item\label{boundary-to-diagonal}
For every $\xi\in\rand_kX$ there exists a constant $a_\xi$ such that
$|k_\xi(x)| \leq a_\xi k(x,x), \forall x\in X.$
\item\label{constants-inclusion}
There is a constant $b>0$ such that $|k(x,y)|>b$ for all $x,y\in X$.
\item\label{isolated-singularity}
Let $\xi\in\rand_kX$ and $(x_n)$ be a sequence in $X$. If $|k_\xi(x_n)|\to\infty$,
then $x_n\xto{k}\xi$.
\item\label{compact-boundary}
If for a sequence $(x_n)$ in $X$ $k(x_n,x_n)\to\infty$ as $n\to\infty$,
then there is a subsequence such that $x_{n_j}\xto{k}\xi\in\rand_kX$.
\end{enumerate}
With these assumptions we prove the following result, which actually is a special
case of our main theorem.
\begin{theorem1}
Let $k$ be a reproducing kernel on a nonempty set $X$ satisfying
\ref{boundary-to-diagonal}-\ref{compact-boundary}.
Let $\phi:X\to X$ be a composition factor of $k$.
Let $\xi\in\partial_kX$ and
suppose that $\Gamma_k(M,\xi)$ approaches $\xi$ for some $M>0$, then the
following are equivalent:
\begin{enumerate}[label=(\roman*)]
\item \[
	\liminf_{x\xto{k}\xi} \frac{k(x,x)}{k(\phi(x),\phi(x))} < \infty.
\]
\item There is a $\lambda\in\partial_kX$ such that the function defined by \[
	q_\xi(x) = \frac{k_\xi(x)}{k_\lambda(\phi(x))}
\]
belongs to $\Hilbert(\frac{k}{k\circ\phi})$.
\item All functions in $\Hilbert(\frac{k}{k\circ\phi})$ have a $\Gamma_k$-limit at $\xi$.
\end{enumerate}
Furthermore, if any of the above hold, then $\phi$ has $E_k$-limit $\lambda$ at $\xi$, and
we have the estimate $0<\|q_\xi\|^2\leq c\leq (Ma_\lambda\|q_\xi\|)^2$ for every $M$ such that
$\Gamma_k(M,\xi)$ approaches $\xi$.
\end{theorem1}
\noindent When $k$ is the Szeg\H{o} kernel on $\disk$ then we are exactly in the situation of Sarason's
theorem.
Many reproducing kernels satisfy \ref{boundary-to-diagonal}-\ref{compact-boundary},
some simple examples include the Drury-Arveson kernel, even in infinite dimensions,
and the Szeg\H{o} kernel on the polydisk. It seems that the most restrictive condition is
the assumption that $\phi$ is a composition factor of $k$. Our main theorem relaxes
this constraint considerably by allowing for distinct reproducing kernels in the numerator and
denominator of \eqref{eq:composition-factor}. This idea is interesting in its own right
as the following example shows.

For $0<\alpha\leq1$, let $\dirichlet_\alpha$ be the standard weighted Dirichlet space that
has reproducing kernel $(1-z\overline w)^{-\alpha},\,z,w\in\disk$. We improve the classical
theorem for the class of contractive multipliers of $\dirichlet_\alpha$. Notably
every contractive multiplier of the classical Dirichlet space is in this class for
every $\alpha>0$.
\addtocounter{corollary}{1}
\begin{corollary}
Fix $\zeta\in\torus$ and let $\phi$ be a contractive multiplier of $\dirichlet_\alpha$, then
the following are equivalent: \begin{enumerate}[label=(\roman*)]
\item \[
	\liminf_{z\to\zeta}\frac{1-|\phi(z)|^2}{(1-|z|^2)^\alpha} = c < \infty.
\]
\item There is a $\lambda\in\torus$ such that the difference quotient \[
	q_\zeta(z) = \frac{1-\phi(z)\overline\lambda}{(1-z\overline\zeta)^\alpha}
\]
has a nontangential limit $c$ at $\zeta$.
\item $\phi$ and the function $z\mapsto \phi'(z)(1-z\overline\zeta)^{1-\alpha}$ has nontangential
limits $\lambda\in\torus$ and $c\lambda\overline\zeta\alpha$ at $\zeta$, respectively.
\end{enumerate}
Furthermore, $c>0$ and $\phi$ has horocyclic limit $\lambda$ at $\zeta$.
\end{corollary}
\addtocounter{corollary}{-2}
\noindent In fact, the same corollary holds in the unit ball, see \cref{section:besov},
of $\plane^d$, which for $\alpha=1$ recovers the result of Jury mentioned above.
Note that for $\alpha<1$, the growth condition (i) is weaker than the classical
Julia-Carathéodory theorem. Therefore the angular derivative does not
exist at this point. However (iii) gives the exact compensating factor
$\phi'$ needs for it to have a nonzero boundary value.
This corollary is one motivation for the general form of our main theorem, which is given
in \S3.

Sarason also proved Julia's lemma as a consequence of his approach.
Its extension can be found in \S4. In addition this yields the following
result regarding iteration of composition factors.
\begin{theorem2}
Let $k$ be a reproducing kernel on a nonempty set $X$ that satisfies
\ref{boundary-to-diagonal}-\ref{compact-boundary}.
If a composition factor $\phi$ of $k$ and $\xi\in\partial_kX$ satisfies one of
\ref{item:liminf-condition}-\ref{item:uniform-boundary-values} in \cref{theorem:main}
with $c < 1$ and $\phi$ has $E_k$-limit $\xi$ at $\xi$, then the sequence of
iterates $\phi^n(x)\xto{k}\xi$ for every $x\in X$.
\\ In particular $\phi$ has no fixed point in $X$.
\end{theorem2}
\noindent The above can be seen as one direction of the Denjoy-Wolff theorem, disregarding the case $c=1$,
see \cref{question:iteration}. Most generalization of this theorem suppose that $\phi$
is fixed point free within $X$ and then prove that the iterates converge. The above results
asserts that $\phi$ is fixed point free.

The paper is organized as follows. \S2 serves a preliminary cause and
introduces the abstractions required. \S2.1 briefly establishes the fundamental theory of
reproducing kernels. In \S2.2 we introduce composition factors of reproducing
kernels. \S2.3 discusses and prove some properties of the reproductive boundary.
The approach regions to the reproductive boundary are realized in \S2.4.
In \S2.5 a list of examples is given to illustrate these three concepts.
\S3 is devoted to stating and proving our main theorem,
the general Julia-Carathéodory theorem for composition factors.
\S4 contains applications of this theorem. \S4.1 contains a general
version of Julia's lemma and proves the  uniqueness of some Denjoy-Wolff points
along with a condition for their existence. In \S4.2 we prove Corollary 2 stated
above, and in \S4.3 we briefly handle the case of the unit ball. The paper
concludes with \S5, a list of questions in the circle of problems introduced in
the present paper.

\section{Preliminaries}

\subsection{Reproducing kernels}
The contents of this section is essentially included in \cite{aronszajn} and \cite{paulsen}.
But we give a brief overview of the theory and state the theorems that we require.

\begin{definition}
Let $X$ be a nonempty set. A function $k:X\times X\to\plane$ is a \emph{reproducing kernel}
if for every finite subset $\{x_1,\ldots,x_n\}\subset X$ the matrix $[k(x_i,x_j)]_{i,j=1}^n$
is positive semidefinite. We write $k\geq0$ or $k(x,y)\geq0$ to mean this.
\\ If there is a a point $0_X\in X$ such that $k(x,0_X)\equiv1$, then we say that $k$ is
normalized at $0_X$.
\end{definition}
By Moore's theorem there exists a unique Hilbert space of functions $\Hilbert(k)$ such
that each \emph{kernel function} $k_y(x) \coloneq x\mapsto k(x,y),\,x,y\in X$ belong to
$\Hilbert(k)$ and they satisfy \[
	f(y) = \la f,k_y\ra,
\]
for every $f\in\Hilbert(k)$. It is an immediate consequence that $\|k_y\|^2=k(y,y)$.
Furthermore, a function $f:X\to\plane$ belongs to $\Hilbert(k)$ if and only if there
is a constant $c>0$ such that \begin{equation}\label{inclusion-criterion}
	c^2k(x,y) - f(x)\overline{f(y)} \geq 0,
\end{equation}
and the least such constant is $\|f\|$, the norm of $f$.

Of particular interest is the following two classes of symbols, in \S2.2 we shall
introduce a third class, which we will compare to these.
\begin{definition}
We call a function $\phi:X\to\plane$
a \emph{multiplier} of $\Hilbert(k)$ if the function $x\mapsto\phi(x)f(x)$ belongs to
$\Hilbert(k)$ whenever $f\in\Hilbert(k)$, we write the set of these as $\Mult(k)$.
\end{definition}
Using the closed graph theorem one can prove that the linear operator
$f\mapsto \phi f$ is bounded on $\Hilbert(k)$.  Furthermore $\phi\in\Mult(k)$ if and only
if there exists a $c>0$ such that \begin{equation}
	c^2k(x,y) - \phi(x)\overline{\phi(y)}k(x,y) \geq 0,
\end{equation}
and the least such constant is the norm of the induced operator. We will
write $\Mult_1(k)$ to mean the set of contractive multipliers on $\Hilbert(k)$.
Often we are interested in excluding the unimodular constants $x\to\lambda,\,|\lambda|=1$
from the set $\Mult_1(k)$, in this case we will write $\Mult_1(k)\setminus\torus$.

\begin{definition}
We call a selfmap $\phi:X\to X$
a \emph{composition symbol} of $\Hilbert(k)$ if the linear operator $C_\phi\coloneq f\mapsto f\circ\phi$
acts on $\Hilbert(k)$. The set of these are denoted by $\Comp(k)$.
\end{definition}
$\phi\in\Comp(k)$
if and only if there exists a constant $c>0$ such that \begin{equation}
	c^2k(x,y) - k(\phi(x),\phi(y)) \geq 0,
\end{equation}
and the least such constant is $\|C_\phi\|$. We will write $\Comp_1(k)$ to mean the set
of contractive composition symbols.
\begin{remark}
The notation $\Comp(k)$ is not standard, but we adopt it since we are concerned only
with the symbols themselves, not the operators that they induce.
\end{remark}

We end this section by stating a few results that are essential to our methods.
\begin{theorem*}[Schur's Product Theorem]
Let $k$ and $t$ be reproducing kernels on a nonempty set $X$, then their product
$x,y\mapsto k(x,y)t(x,y)$ is as well.
\end{theorem*}

\begin{proposition}
Let $X$ and $Y$ be two nonempty sets and let $k$ be a reproducing kernel on $X$.
For any map $\phi:Y\to X$ the composed kernel $k\circ\phi = x,y\mapsto k(\phi(x),\phi(y))$
is a reproducing kernel on $Y$.
\end{proposition}

The following fact, together with weak compactness will be used throughout the paper.
It is an immediate consequence of the density of the linear span of $\{k_x:x\in X\}$ in
$\Hilbert(k)$.
\begin{proposition}
A sequence $(f_n)$ in $\Hilbert(k)$ converges weakly to a function $f\in\Hilbert(k)$
if and only if $\|f_n\|$ is bounded and $f_n\to f$ pointwise.
\end{proposition}

\subsection{Composition factors}
Throughout this subsection fix $k$ and $t$ to be reproducing kernels on nonempty
sets $X$ and $Y$, respectively. In addition, we suppose that $t$ is \emph{nonzero}
i.e. $t(x,y)\neq0$ for all $x,y\in Y$.
\begin{definition}
Let $\phi:X\to Y$ be a map. If \begin{equation}
	\frac{k}{t\circ\phi} \geq 0,
\end{equation}
then we call $\phi$ a \emph{composition factor} of $k$ with $t$. We denote the
set of these as $\Fact(k, t)$.
\\ In the special case that $k=t$ then we simply say that $\phi$ is a composition
factor of $k$ and write the set as $\Fact(k)$.
\end{definition}

\begin{proposition}\label{proposition:weighted-composition}
Let $\phi\in\Fact(k,t)$, then for every $f\in\Hilbert(\frac{k}{t\circ\phi})$ the
pair $(f,\phi)$ induces a bounded weighted composition operator from $\Hilbert(t)$
into $\Hilbert(k)$ of norm at most $\|f\|$. That is, the linear map defined by
$(W_{f,\phi})g(x) = f(x)g(\phi(x))$ is bounded from $\Hilbert(t)$ into $\Hilbert(k)$.
\\ In particular, if the constant function $x\mapsto1$ belongs to $\Hilbert(\frac{k}{t\circ\phi})$,
then $\phi\in\Comp(t, k)$ and
if $x\mapsto1\in\Hilbert(t\circ\phi)$, then $\Hilbert(\frac{k}{t\circ\phi})\subset\Hilbert(k)$.
\end{proposition}
\begin{proof}
Let $f\in\Hilbert(\frac{k}{t\circ\phi})$, this is equivalent to \[
	\|f\|^2\frac{k(x,y)}{t\circ\phi(x,y)} - f(x)\overline{f(y)} \geq 0.
\]
Multiplying the above kernel with $t\circ\phi$ yields, by Schur's Product Theorem, \[
	\|f\|^2k(x,y) - f(x)\overline{f(y)}t\circ\phi(x,y) \geq 0,
\]
which is equivalent to our desired conclusion.
\end{proof}
\begin{remark}
Suppose that $k$ and $t$ are normalized at some points $0_X\in X$ and $0_Y\in Y$,
respectively. If $\phi\in\Fact(k,t)$ with $\phi(0_X) = 0_Y$, then the kernel
$\frac{k}{t\circ\phi}$ is normalized at $0_X$. In this case,
\cref{proposition:weighted-composition} shows that $\phi\in\Comp(k,t)$ with norm at most 1.
\end{remark}

\begin{proposition}\label{proposition:transitive-composition}
Let $r$ be a nonzero reproducing kernel on a nonempty set $Z$.
If $\phi\in\Fact(k,t)$ and $\psi\in\Fact(t,r)$, then $\psi\circ\phi\in\Fact(k,r)$.
\end{proposition}
\begin{proof}
Compose the kernel $\frac{t}{r\circ\psi}$ with $\phi$ and multiply with $\frac{k}{t\circ\phi}$,
the result then follows by Schur's Product Theorem.
\end{proof}
\begin{remark}
In the case that $t=k$ \cref{proposition:transitive-composition} shows that $\Fact(k)$ forms
a semigroup under composition.
\end{remark}

\begin{proposition}\label{proposition:automorphism-composition}
Let $k$ be normalized at a point $0_X$. If for $\phi\in\Fact(k)$ there exists an automorphism
$\psi$ of $X$ with $\psi(\phi(0_X)) = 0_X$ such that $\psi\in\Fact(k)\cap\Comp(k)$ and
$\psi^{-1}\in\Comp(k)$, then $\phi\in\Comp(k)$.
\end{proposition}
\begin{proof}
By \cref{proposition:transitive-composition}, $\psi\circ\phi\in\Fact(k)$ and
the quotient kernel $\frac{k}{k\circ\psi\circ\phi}$ is normalized at $0_X$. Therefore
$\psi\circ\phi\in\Comp(k)$ by \cref{proposition:weighted-composition} and the result
follows by composing $\psi^{-1}$ with $\psi\circ\phi$.
\end{proof}

\subsection{Reproductive boundary}
Throughout this section, let $k$ be a reproducing kernel on a nonempty set $X$.
Recall the definition of the reproductive boundary from the introduction.
\begin{definition}
Let $\widehat X$ be the set of sequences $\hat x=(x_n)$ of $X$ such that the sequence
of kernel functions $k_{x_n}$ converge pointwise. Two sequences $\hat x,\hat y\in\widehat X$
are equivalent, $\hat x\sim\hat y$ if and only if the corresponding sequences of functions
$(k_{x_n})$ and $(k_{y_n})$ converge pointwise to the same function; we let $\widetilde X$
denote the quotient set. For $\xi\in\widetilde X$ we denote the above limiting function
as $k_\xi$. For these functions the following possibilities can occur:
\begin{enumerate}[label=(\roman*)]
\item $k_\xi = k_x$ for some $x\in X$,
\item $k_\xi\in\Hilbert(k)$ but (i) is fails,
\item $k_\xi\notin\Hilbert(k)$.
\end{enumerate}
We define the \emph{reproductive boundary} to be the set of points $\xi\in\widetilde X$
such that (iii), written as $\partial_kX$, and call their corresponding functions
\emph{boundary kernel functions} or \emph{boundary functions} for short.
\\ Similarly we define the \emph{reproductive interior} to be the set of point $\xi\in\widetilde X$
such that \emph{either} (i) or (ii), written as $\Int_kX$, and call their corresponding
functions \emph{interior kernel functions} or \emph{interior functions} for short.
\\ In any case, if the sequence of functions $k_{x_n}$ tend pointwise to $k_\xi$,
then we say that the sequence is $k$-convergent and write $x_n\xto{k}\xi$, in other words
the equivalence class $[(x_n)]_\sim = \xi$. Similarly
we write $x\xto{k}\xi$ to mean \emph{along any sequence $(x_n)$ such that $x\xto{k}\xi$}.
\end{definition}
The reproductive boundary is our sole concern in this paper, we leave the discussion
of the reproductive interior as the following remark.
\begin{remark}
If (i), then we may freely identify $\xi$ with a point $x\in X$, loosely speaking
$\xi$ is an interior point. However, if $k$ does not separate points, then the point $x$ might not be unique.
The possibility (ii) is never reached for the Szeg\H{o} kernel, in a sense this kernel is
complete. Consider restricting the Szeg\H{o} kernel $k^*(z,w)=(1-z\bar w)^{-1}$
to the punctured plane $\disk^* = \disk\setminus\{0\}$, in this case the constant function
$z\mapsto1$ belongs to $\Hilbert(k^*)$, but is not equal to $k^*_w$ for any $w\in\disk^*$.
In general we have the above trichotomy, but we can reduce it to a dichotomy by letting
$\bar k(\xi,\upsilon) = \la k_\upsilon, k_\xi\ra_{\Hilbert(k)},\,\xi,\upsilon\in\Int_kX$
and by embedding $\Hilbert(k)$ into the space $\Hilbert(\bar k)$, i.e. we can extend every function $f\in\Hilbert(k)$ to
$\Int_kX$ via the rule $f(\xi) = \la f, k_\xi\ra_{\Hilbert(k)}$.
\end{remark}
\begin{remark}\label{remark:topology}
In general, $\widetilde X$ is not a topological space. One could let $\{k_x:x\in X\}$
inherit the topology of pointwise convergence. But for the purposes of
this paper it is more convenient to work with sequences rather than nets.
Nonetheless, in the worked examples there will often be a clear correspondence
to a preexisting topology. \\ Another approach that might be of independent
interest is as follows. Consider the closure of $\{k_x:x\in X\}$ in the topology
of uniform convergence on compact sets given by the metric, see \cite{agler-mccarthy_cnp}
\begin{equation}\label{equation:p-metric}
	p_k(x,y)
	= \sqrt{1-\frac{|k(x,y)|^2}{k(x,x)^\frac12 k(y,y)^\frac12}}, \qquad x,y\in X.
\end{equation}
This construction is similar to the Busemann boundary, see \cite{bracci_boundary}.
\\ When is this topology metrizable?
Let $(x_n)$ be a sequence in $X$ and consider the properties: \begin{align}
	&\text{if $k_{x_n}\xto{w}g\in\Hilbert(k)$, then $g=k_x$ for some $x\in X$,} \\
	&\text{if $\|k_{x_n}\|\to\infty$, then $\hat k_{x_n} = \frac{k_{x_n}}{\|k_{x_n}\|}\xto{w}0$,} \\
	&\text{if $k_{x_n}\xto{w}k_x$, then $k_{x_n}\to k_x$ in norm.}
\end{align}
If $k$ satisfies the first two, then $X$ is complete under the $p_k$ metric.
If it also satisfies the third, then the closed balls centered around some
point $x\in X$ given by \[
	\overline{B(\epsilon,x)} = \{y\in X:p_k(x,y)\leq\epsilon\}, \qquad 0<\epsilon<1
\]
are compact. Clearly $\cup_{0<\epsilon<1}\overline{B(x,\epsilon)} = X$ so $X$ is
hemicompact with respect to $p_k$. Hence the closure of $\{k_x:x\}$ in the
compact-open topology is metrizable.
\end{remark}

The Szegő kernel $k(z,w)=(1-z\overline w)^{-1}$ on the unit disk satisfies
$k(z,z)\to\infty$ as $|z|\to1^-$. Indeed, any kernel evaluated on the diagonal
diverges to $\infty$ as the variable tends to the reproductive boundary.
\begin{lemma}\label{lemma:boundary-explosion}
Let $x_n\xto{k}\xi\in\rand_kX$, then $\|k_{x_n}\|^2 = k(x_n,x_n)\to\infty$ as
$n\to\infty$.
\end{lemma}
\begin{proof}
Suppose on the contrary that $k(x_n,x_n)$ is bounded. By weak compactness,
passing to a subsequence if necessary, we can assume that $k_{x_n}$ converges
weakly to a function $g\in\Hilbert(k)$. But \[
	g(y)
	= \la g,k_y\ra
	= \lim_{n\to\infty}\la k_{x_n},k_y\ra
	= \lim_{n\to\infty}k_{x_n}(y)
	= k_\xi(y),
\]
contradicting the fact that $k_\xi\notin\Hilbert(k)$.
\end{proof}
\begin{remark}
It appears that \cref{lemma:boundary-explosion} suggests an equivalent way to
define the boundary. Suppose conversely that for a sequence $(x_n)$ the
functions $k_{x_n}$ converge pointwise to a function $k_\xi$ and
$k(x_n,x_n)\to\infty$, does this imply that $\xi\in\rand_kX$? The following
example shows that this fails in general.  Observe the infinite matrix \[
	\begin{pmatrix}
	1 & 1 & 1 & \ldots \\
	1 & 2 & 1 & \ldots \\
	1 & 1 & 3 & \ldots \\
	\vdots & \vdots & \vdots & \ddots
	\end{pmatrix},
\]
because it has strictly positive subdeterminants,
it can be regarded as a reproducing kernel on the set of natural numbers $\natural$.
For any sequence $(j_n)$ in $\natural$ such that $j_n\to\infty$ as $n\to\infty$, then
$k(j_n,j_n)=j_n\to\infty$. But the functions $i\mapsto k(i,j_n)$ tend pointwise
to $1$. In other words they tend pointwise to the function $i\mapsto k(i,0)$, represented
by the top row, which belongs to $\Hilbert(k)$.
\end{remark}

It is a well known fact that functions $f\in H^2$, the Hardy space on the unit disk,
satisfy the growth restriction $f(z) = o((1-|z|^2)^\frac12)$ as $|z|\to1^-$.
We have the following generalization to functions in $\Hilbert(k)$ when the
variable tends to the reproductive boundary of $k$.
\begin{lemma}\label{lemma:growth-restriction}
If $x_n\xto{k}\xi\in\rand_kX$, then the normalized kernel functions
$\hat k_x = \frac{k_x}{k(x,x)^\frac12}$ tend weakly to $0$ along this sequence,
that is, for every $f\in\Hilbert(k)$ \begin{equation}
	\lim_{n\to\infty}\la f,\hat k_{x_n}\ra
	= \lim_{n\to\infty}\frac{f(x_n)}{k(x_n,x_n)^\frac12}
	= 0.
\end{equation}
\end{lemma}
\begin{proof}
Since $\|\hat k_{x_n}\|=1$ it is sufficient to prove that $\hat k_{x_n}$ converge
pointwise to $0$. The sequence $k_{x_n}(y)$ is convergent by assumption, hence
it is bounded, therefore \[
	\hat k_{x_n}(y) = \frac{k_{x_n}(y)}{k(x_n,x_n)^\frac12}
\]
tend to $0$ as $n\to\infty$ by \cref{lemma:boundary-explosion}.
\end{proof}

\subsection{Approach regions}
\begin{definition}
Let $k$ be a reproducing kernel on a nonempty set $X$.
For $\xi\in\rand_k X$ and $M>0$ consider the sets \begin{align}
	\Gamma_k(M,\xi) &= \{x\in X: k(x,x)\leq M|k_\xi(x)|\}, \\
	E_k(M,\xi)      &= \{x\in X: k(x,x)\leq M|k_\xi(x)|^2\}.
\end{align}
If a sequence $x_n\xto{k}\xi$ within a fixed set $\Gamma_k(M,\xi)$, then we say that
$x_n$ approaches $\xi$ $\Gamma_k$-wise and write $x_n\xto{\Gamma_k}\xi$. If $M$ is
small, then there might be no sequence in $\Gamma_k(\xi,M)$ that is $k$-convergent to $\xi$,
but if there is at least one such sequence, then we say that $\Gamma_k(M,\xi)$
\emph{approaches} $\xi$.
\\ Let $f:X\to\plane$, if for every sequence $x_n\xto{\Gamma_k}\xi$ we have
$\lim_{n\to\infty}f(x_n)=\lambda\in\plane$, then we say that $f$ has $\Gamma_k$-limit $\lambda$
at $\xi$. Of course the same definition is valid for any metric space in place of $\plane$.
\\ Let $t$ be a reproducing kernel on a nonempty set $Y$. For a map $\phi:Y\to X$ we will
similarly say that $\phi$ has a $\Gamma_k$-limit $\lambda\in Y\cup\rand_t Y$ at $\xi$ if
$\phi(x_n)\xto{t}\lambda$ as $n\to\infty$ for every sequence $x_n\xto{k}\xi$.
\end{definition}
\begin{remark}
It might not be true that $\Gamma_k(M,\xi)\subset E_k(M,\xi)$ however, for a
sequence $x_n\xto{k}\xi$ within $\Gamma_k(M,\xi)$ we know that
$k(x_n,x_n)\to\infty$ by \cref{lemma:boundary-explosion}. So  if $n$ is big enough to guarantee $k(x_n,x_n)>M$,
then $|k_\xi(x_n)|>1$ and therefore, for these $n$ we also have
$x_n\in E_k(M,\xi)$. In conclusion, the notion of $E_k$-limit is stronger
than $\Gamma_k$-limit.
\end{remark}

\subsection{Examples}
Here we list examples illustrating the reproductive boundary, approach regions and composition
factors of different reproducing kernels.

\begin{example}\label{example:szego}
Let $k(z,w) = (1-z\overline w)^{-1}$ be the Szegő kernel on $\disk$ corresponding to the
Hardy space $H^2$. Since $\frac1{1-z}$ has
an isolated pole at $1$ one easily sees that $\rand_k\disk$ can be identified with
$\torus$. Moreover, any sequence $(z_n)$ in $\disk$ converges to a point $\zeta\in\torus$
if and only if $z_n\xto{k}\zeta$.
\\ The approach regions $\Gamma_k$ and $E_k$ are seen to be the families of Stolz sectors
and horocycles, respectively, which in turn correspond to the notion of nontangential
and horocyclic boundary limits.
\\ A selfmap $\phi:\disk\to\disk$ belongs to $\Fact(k)$ if and only if $\phi$ is a
contractive multiplier of $H^2$ and is not a unimodular constant; the set of these are
precisely the analytic selfmaps of $\disk$.
The analytic automorphism of the disk $\psi_a$ that interchange $a$ and $0$ is
given by $\psi_a(z) = (a-z)/(1-z\overline a)$ and they
satisfy \begin{equation}\label{equation:disk-automorphism}
	1 - \psi_a(z)\overline{\psi_a(w)}
	= \frac{(1-|a|^2)(1-z\overline w)}{(1-z\overline a)(1-a\overline w)}
\end{equation}
A computation reveals that $\psi_a\in\Comp(k)$, therefore \cref{proposition:weighted-composition}
shows that $\Mult_1(k)\setminus\torus=\Comp(k)=\Fact(k)$.
\end{example}

\begin{example}[\cref{example:szego} continued]\label{example:szego-power}
Let $k(z,w) = \frac1{(1-z\overline w)^\alpha},\,\alpha>0$. For $\alpha<1$, $\alpha=1$, and
$\alpha>1$, then $\Hilbert(k)$ is a weighted Dirichlet space, see \cref{section:dirichlet},
$H^2$, and a standard weighted Bergman space, respectively.
The same considerations \cref{example:szego} hold, and the approach regions with respect
to $k$ are exactly the same for any $\alpha$. By \eqref{equation:disk-automorphism} one
sees that each analytic disk automorphism belongs to $\Fact(k)$ and also it is routine
to show that $\psi_a\in\Comp(k)$, hence \cref{proposition:weighted-composition} proves
that $\Fact(k)\subset\Comp(k)$.

For $\alpha\geq1$, then $\Fact(k)$ are precisely the analytic selfmaps of $\disk$.
While for $\alpha<1$, then they are strictly contained in this set. Letting $\alpha=\frac12$
and set $\phi(z)=z^2$, then \[
	\left(\frac{1-z^2\overline w^2}{1-z\overline w}\right)^\frac12
	= \sqrt{1+z\overline w},
\]
which is not a reproducing kernel, since expressed as a power series in $z\overline w$ it
has some negative coefficients. But we can say a few things. By \cref{proposition:transitive-composition},
each contractive multiplier of $\Hilbert(k)$ belong to $\Fact(k)$, which in particular
includes the contractive multipliers of the classical Dirichlet space. There is another
class of functions that belong to $\Fact(k)$. For any analytic selfmap of $\disk$ we
have \[
	\frac{1-\phi(z)\overline{\phi(w)}}{1-z\overline w} \geq 0.
\]
Chu \cite{chu_pick} classified exactly for which $\phi$ the above kernel
is complete Nevanlinna-Pick, see \cite{agler-mccarthy_cnp}, which is equivalent
to being able to write it as \[
	\frac1{1-\la u(z),u(w)\ra},
\]
where $u:X\to\ell^2$. Then by Schur's Product Theorem and the power series representation
of $(1-x)^{-\alpha}$ we obtain \[
	\frac{k(z,w)}{k\circ\phi(z,w)}
	= \left(\frac1{1-\la u(z),u(w)\ra}\right)^\alpha
	\geq 0.
\]
But by no means should this condition be necessary.
\end{example}

\begin{example}
Let $k(z,w) = \frac1{z\overline w}\log(\frac1{1-z\overline w}),\,z,w\in\disk$ be the reproducing
kernel for $\dirichlet$, the Dirichlet space on the unit disk, see
\cite{ross_dirichlet}. We have $\rand_k\disk=\torus$ and for $\zeta\in\torus$
we have the family approach regions equivalent to the two inequalities
\begin{align*}
	|\zeta - z| &\leq (1-|z|^2)^\frac1M, \\
	|\zeta - z| &\leq e^{-\sqrt{\frac1M\log{\frac1{1-|z|^2}}}},
\end{align*}
corresponding to $\Gamma_k$ and $E_k$, respectively. Kinney \cite{kinney_dirichlet}
showed that every function in $\dirichlet$ has $\Gamma_k$-limits almost
everywhere on $\torus$. This was later improved by
Nagel, Rudin, and Shapiro \cite{nagel_dirichlet} to the exponential approach
region defined by the inequality \[
	|\zeta - z| \leq M\log\left(\frac1{1-|z|^2}\right)^{-1}.
\]
In terms of reproducing kernels this compares the Szegő kernel and the Dirichlet kernel.
\end{example}

\begin{example}
Let $k(x,y) = \min(x,y)$ be defined on $[0,\infty)$, this reproducing kernel corresponds
to the absolutely continuous functions $f:[0,\infty)\to\reals$ with $f(0) = 0$ such that \[
	\int_0^\infty |f'(x)|^2dx < \infty,
\]
 given the inner product \[
	\la f,g\ra = \int_0^\infty f'(x)\overline{g'(x)}dx.
\]
Let $x_n\xto{k}\xi$, if the sequence $(x_n)$ is bounded, then the function $k_\xi$ is
given by $k_x$ for some $x\in[0,\infty)$. In the other case that $x_n\to\infty$, then
$k_{x_n}$ converges pointwise to the identity function $x\mapsto x$, which is not included
in $\Hilbert(k)$.
\\ Therefore, the reproductive boundary is the singleton containing the identity map $\{x\mapsto x\}$,
which we can identify as a point at $\infty$. By \cref{lemma:growth-restriction} each function
$f\in\Hilbert(k)$ satisfies $\frac{f(x)}{\sqrt{x}}\to0$ as $x\to\infty$.
\end{example}

\begin{example}
Let $k(z,w) = \frac{1-b(z)\overline{b(w)}}{1-z\overline w}$ be the de Branges-Rovnyak space with
$b(z) = \frac{z+1}2$. It is easily shown that if $w\to1$, then the functions $k_w$
converge pointwise to the constant function $z\mapsto\frac12$, which belongs to $\Hilbert(k)$,
therefore $\rand_k\disk = \torus\setminus\{1\}$.
\end{example}

\begin{example}\label{example:riemann-zeta}
Let $\zeta(s) = \sum_{n=1}^\infty n^{-s},\,\Re{s}>1$ be the Riemann zeta function and
consider its analytic continuation to $\plane\setminus\{1\}$.
Set $\plane_\frac12 = \{z\in\plane:\Re{z}>\frac12\}$ and
let $k(z,w) = \zeta(z + \overline w),\,z,w\in\plane_\frac12$, this reproducing kernel
correspond to the Hardy space of Dirichlet series, see \cite{hedenmalm-seip_hardy}.
In this case $\rand_k\plane_\frac12 = \{z\in\plane:\Re{z}=\frac12\}$, the critical
line. We lack geometrical descriptions of the $\Gamma_k$ and $E_k$ approach regions,
see \cref{question:riemann-zeta}.
\end{example}

\begin{example}
By the power series expansion of $e^x$ and Schur's Product Theorem we have that
$e^k\geq0$. A selfmap $\phi:X\to X$ belongs to $\Fact(e^k)$ if and only if \[
	\frac{e^k}{e^{k\circ\phi}} = e^{k-k\circ\phi}\geq0.
\]
Hence $\Fact(e^k)$ includes $\Comp_1(k)$.
\end{example}

\begin{example}
Let $k(z,w) = (1-\la z,w\ra)^\alpha,\,\alpha>0$ be a power of the Drury-Arveson
kernel on the unit ball $\ball_d$ of $\plane^d$, see \cite{hartz_drury-arveson}.
Then $\rand_k\ball_d = \rand\ball_d$ and the $\Gamma_k$ and $E_k$ approach regions
are the Korányi regions and horospheres, respectively. See \cref{section:besov} for
more details. Just as in $\disk$ one can prove that $\Fact(k)\subset\Comp(k)$.
\end{example}

\begin{example}
For $d>0$ let \[
	k(z,w) = \prod_{j=1}^d\frac1{1-z_j\overline w_j}, \qquad z,w\in\disk^d,
\]
be the reproducing kernel corresponding to the Hardy space on the polydisk $\disk^d$.
The same considerations as in \cref{example:szego} show that $\rand_k\disk^d=\rand\disk^d$,
the non distinguished boundary.
Suppose that $\phi(z) = \sigma(\phi_1(z_1),\ldots,\phi_d(z_d))$, where each
$\phi_j:\disk\to\disk$ and $\sigma$ is a permutation of $\{1,\ldots,d\}$.
Then $\phi\in\Fact(k)$ by Schur's Product Theorem, since each $\phi_j$ satisfies \[
	\frac{1-\phi_j(z_j)\overline{\phi_j(w_j)}}{1-z_j\overline w_j} \geq0.
\]
The analytic automorphisms of $\disk^d$ are given by $d$ automorphisms of each separate
disk of $\disk^d$ followed by a permutation \cite{rudin_polydisk}. Hence they
belong to $\Fact(k)$. Furthermore, they are also included in $\Comp(k)$ so therefore
$\Fact(k)\subset\Comp(k)$, a fact noted by Chu in the bidisk \cite{chu_bidisk}, but
he proved it more directly. The reverse inclusion is false, for example $z\mapsto(z_1,z_1)$
belongs to $\Comp(k)$ when $d=2$, but clearly \[
	\frac{1-z_1\overline w_1}{1-z_2\overline w_2} \ngeq 0.
\]
\end{example}

\begin{example}
If every nonconstant $\psi\in\Mult_1(k)$ satisfies $|\psi(x)|<1$, for example $\Hilbert(k)$ is a space
of analytic functions, then $\Mult_1(k)\setminus\torus=\Fact(k,t)$ where
$t(z,w)=(1-z\overline w)^{-1}$ is the Szegő kernel on $\disk$.
Furthermore, in this case \cref{proposition:transitive-composition} shows that if $\phi\in\Fact(k)$,
then $\psi\circ\phi\in\Mult_1(k)$.

Substituting $t$ for the weighted Bergman kernel $t(z,w)=(1-z\overline w)^{-n}$ one can identify
$\Fact(k,t)$ as the $n$-contractive multipliers on $\Hilbert(k)$ except $\torus$, see \cite{gu_higher-order}.
\end{example}

\begin{example}
If $\Hilbert(k)$ is a space of analytic functions on a domain of $\plane$ that is shift invariant,
that is, the identity function $z\mapsto z$ is in $\Mult(k)$, then $\Fact(k)\subset\Mult(k)$.
\end{example}

\section{Main theorem}
As previously mentioned, our main theorem requires some regularity assumptions. We restate
them here for the convenience of the reader.
\\ Let $t$ be a reproducing kernel on a nonempty set $Y$.
\begin{enumerate}[label=(\Alph*)]
\item
For every $\lambda\in\rand_tY$ there exists a constant $a_\lambda$ such that
$|t_\lambda(y)| \leq a_\lambda t(y,y), \,\forall y\in Y.$
\item
There is a constant $b>0$ such that $|t(x,y)|>b$ for all $x,y\in Y$.
\item
Let $\lambda\in\rand_tY$ and $(y_n)$ be a sequence in $Y$. If $|t_\lambda(y_n)|\to\infty$,
then $y_n\xto{t}\lambda$.
\item
If for a sequence $(y_n)$ in $Y$ $t(y_n,y_n)\to\infty$ as $n\to\infty$,
then there is a subsequence such that $y_{n_j}\xto{t}\lambda\in\rand_tY$.
\end{enumerate}
\begin{remark}
For every function $f\in\Hilbert(t)$ we have the estimate $|f(y)|\leq\|f\|k(y,y)^\frac12$,
\ref{boundary-to-diagonal} is a weaker form of this estimate on the boundary.
\ref{constants-inclusion} of course implies that $t$ is nonzero but also that
$|t_\lambda(x)|>0$ for all $x\in Y$ and $\lambda\in\rand_tY$, furthermore
it asserts that the constant function $x\mapsto1$ belongs to $\Hilbert(t)$.
The condition \ref{isolated-singularity} is summarized as each boundary function
having an isolated singularity.
Lastly \ref{compact-boundary} should be viewed as a compactness condition, see the
example below.
\end{remark}
\begin{example}\label{example:szego-conditions}
Let $t(z,w)=(1-z\overline w)^{-1}$ be the Szegő kernel on $\disk$, then for
$\lambda\in\torus$ we have $2|\lambda - z| \geq 1-|z|^2$, hence
\ref{boundary-to-diagonal} is satisfied.
Clearly $|t(z,w)|>1$ fulfilling \ref{constants-inclusion} and each boundary function
$t_\lambda$ has an isolated singularity at $\lambda$, hence it satisfies
\ref{isolated-singularity}.
Lastly $t(y,y)$ tends to $\infty$ if and only if $|y|\to1^-$ so
\ref{compact-boundary} follows by the compactness of $\torus$.
\end{example}

\begin{theorem}\label{theorem:main}
Let $k$ and $t$ be reproducing kernels on nonempty sets $X$ and $Y$, respectively.
Suppose that $t$ satisfies \ref{boundary-to-diagonal}-\ref{compact-boundary}.
For $\phi\in\Fact(k,t)$ and $\xi\in\rand_kX$, suppose that $\Gamma_k(M,\xi)$
approaches $\xi$ for some $M>0$, then the following are equivalent:
\begin{enumerate}[label=(\roman*)]
\item\label{item:liminf-condition} \[
	\liminf_{x\xto{k}\xi}\frac{k(x,x)}{t\circ\phi(x,x)} = c < \infty.
\]
\item\label{item:kernel-quotient}
There is a $\lambda\in\rand_tY$ such that the function defined by \[
	q_\xi(x) = \frac{k_\xi(x)}{t_\lambda(\phi(x))}
\]
belongs to $\Hilbert(\frac{k}{t\circ\phi})$.
\item\label{item:uniform-boundary-values}
All functions in $\Hilbert(\frac{k}{t\circ\phi})$ have a $\Gamma_k$-limit at $\xi$.
\end{enumerate}
Furthermore, if any of the above hold, $\phi$ has $E_k$-limit $\lambda$ at $\xi$,
and we have the estimate $0<\|q_\xi\|^2\leq c \leq (Ma_\lambda\|q_\xi\|)^2$ for
every $M$ such that $\Gamma_k(M,\xi)$ approaches $\xi$.
\end{theorem}
\begin{proof}
Fix the notation $q=\frac{k}{t\circ{\phi}}$.

Suppose that \ref{item:liminf-condition} then by assumption there exists a
sequence $x_n\xto{k}\xi$ such that $\|q_{x_n}\|^2 = q(x_n,x_n)\to c$. By
\cref{lemma:boundary-explosion}, $k(x_n,x_n)\to\infty$ as $n\to\infty$,
therefore $t\circ\phi(x_n,x_n)\to\infty$ as $n\to\infty$. Hence, by weak
compactness and \ref{compact-boundary}, passing to a subsequence if
necessary, there is a function $q_\xi\in\Hilbert(q)$ such that $q_{x_n}\xto{w}q_\xi$,
as well as a $\lambda\in\rand_tY$ such that $\phi(x_n)\xto{t}\lambda$. We
compute the function $q_\xi$ as \begin{equation}\label{equation:qxi-computation}
	q_\xi(y)
	= \la q_\xi, q_y\ra
	= \lim_{n\to\infty}\la q_{x_n}, q_y\ra
	= \lim_{n\to\infty}\frac{k_{x_n}(y)}{t(\phi(y),\phi(x_n))}
	= \frac{k_\xi(y)}{t_\lambda(\phi(y))}.
\end{equation}

Suppose that \ref{item:kernel-quotient}, and let $x_n\xto{E_k}\xi$.
\cref{proposition:weighted-composition} together with \ref{constants-inclusion}
says that the function $q_\xi$ belongs to $\Hilbert(k)$. Therefore, by \cref{lemma:growth-restriction} \[
	\frac{q_\xi(x_n)}{k(x_n,x_n)}
	= \frac{k_\xi(x_n)}{t_\lambda(\phi(x_n))k(x_n,x_n)^\frac12}
	\to 0.
\]
But since $x_n\in E_k(M,\xi)$ for some fixed $M>0$, that is,
$\frac{|k_\xi(x_n)|}{k(x_n,x_n)^\frac12}\geq\frac1M$, we have
$|t_\lambda(\phi(x_n))| \to \infty$, hence $\phi(x_n)\xto{t}\lambda$
because of \ref{isolated-singularity}.
\\ Now let $x_n\to\xi$ within $\Gamma_k(M,\xi)$ for some fixed $M>0$,
\ref{item:uniform-boundary-values} is equivalent to showing that the
functions $q_{x_n}$ tend weakly to $q_\xi$ as $n\to\infty$. In the previous
paragraph we established pointwise convergence, since the linear span of
$\{q_x:x\in X\}$ is dense in $\Hilbert(q)$ it is then enough to prove that the
sequence $(q_{x_n})$ is bounded in norm. To this end, we estimate as follows,
for any $x\in\Gamma_k(M,\xi)$ we have \begin{equation}\label{equation:c-norm-estimate}
	\|q_{x}\|^2
	= \frac{k(x,x)}{t\circ\phi(x,x)}
	\leq \frac{|k_\xi(x)|}{t\circ\phi(x,x)}
	\leq Ma_\lambda\frac{|k_\xi(x)|}{|t_\lambda(\phi(x))|}
	\leq Ma_\lambda\|q_\xi\|\|q_x\|,
\end{equation}
where the three inequalities are given by $x\in\Gamma_k(M,\xi)$,
\ref{boundary-to-diagonal}, and Cauchy-Schwarz, respectively.

Suppose that \ref{item:uniform-boundary-values}, as stated previously,
this is equivalent to the sequence $(q_{x_n})$ tending weakly to $q_\xi$
whenever $x_n\xto{\Gamma_k}\xi$. Therefore the implication to \ref{item:liminf-condition}
is given by the principle of uniform boundedness.

The computation \eqref{equation:qxi-computation} reveals that $c\geq\|q_\xi\|^2$, which in
turn satisfies $\|q_\xi\|>0$, for otherwise $q_\xi$ is the zero function so $k_\xi$ is
as well, contradicting the fact that $k_\xi\notin\Hilbert(k)$.
In the proof of $(ii)\implies(iii)$ it was established that $\phi$ has $E_k$-limit
$\lambda$ at $\xi$.
The final desired estimate was given by \eqref{equation:c-norm-estimate}.
\end{proof}

\begin{example}
Let $\ball_d,\,d>1$ be the unit ball in $\plane^d$ and let
$t(z,w) = (1-\la z,w\ra)^{-1},\,z,w\in\plane^d$ be the reproducing kernel
for the Drury-Arveson space. Similar considerations as in \cref{example:szego-conditions}
show that $t$ satisfies \ref{boundary-to-diagonal}-\ref{compact-boundary}. This
holds even for the infinite case when $\ball_\natural$ is the unit ball of $\ell^2$ when considering
weak compactness. Importantly, for a reproducing kernel $k$ on a nonempty set $X$,
the class of functions $\Fact(k,t)$ are precisely the contractive multipliers from
$\Hilbert(k)\otimes\plane^d$ into $\Hilbert(k)$, where $\plane^d$ is replaced by $\ell^2$ in
the infinite dimensional case.
\\ Likewise the conditions are fulfilled for any kernel of the form $t(z,w)^\alpha,\,\alpha>0$.
\end{example}

The example above leads to the following corollary concerning contractive row multipliers.
\begin{corollary}\label{corollary:row-contractive-multiplier}
Let $k$ be a reproducing kernel on a nonempty set $X$. For a contractive multiplier $\phi$
from $\Hilbert(k)\otimes\ell^2$ into $\Hilbert(k)$ and $\xi\in\rand_kX$, suppose that
$\Gamma_k(M,\xi)$ approaches $\xi$ for some $M>0$, then the following are equivalent:
\begin{enumerate}[label=(\roman*)]
\item \[
	\liminf_{x\xto{k}\xi}k(x,x)(1-\|\phi(x)\|^2) = c < \infty.
\]
\item There is a $\lambda\in\ell^2$ with $\|\lambda\|=1$ such that the function defined by \[
	x\mapsto k_\xi(x)(1-\la\phi(x),\lambda\ra)
\]
belongs to $\Hilbert(k(1-\phi\phi^*))$.
\item All functions in $\Hilbert(k(1-\phi\phi^*))$ have a $\Gamma_k$-limit at $\xi$.
\end{enumerate}
Furthermore, if any of the above hold, then $\phi$ has a $E_k$-limit $\lambda$ at $\xi$,
and we have the estimate $0<\|q_\xi\|^2\leq c\leq (Ma_\lambda\|q_\xi\|)^2$ for
every $M$ such that $\Gamma_k(M,\xi)$ approaches $\xi$.
\end{corollary}

\section{Applications}

\subsection{Julia's lemma and iteration in the case where $t = k$}\label{section:iteration}
Firstly we have a general version of Julia's lemma.
\begin{lemma}\label{lemma:julia}
Let $k$ be a reproducing kernel on a nonempty set $X$ that satisfies
\ref{boundary-to-diagonal}-\ref{compact-boundary}. For $\xi\in\rand_kX$
suppose that $\phi\in\Fact(k)$ satisfies one of
\ref{item:liminf-condition}-\ref{item:uniform-boundary-values} in \cref{theorem:main},
denote the $E_k$ boundary value of $\phi$ at $\xi$ by $\phi(\xi)$, then \begin{equation}
	\frac{k(\phi(x),\phi(x))}{|k_{\phi(\xi)}(\phi(x))|^2}
	\leq c\frac{k(x,x)}{|k_\xi(x)|^2},
\end{equation}
where $c$ is the constant in \ref{item:liminf-condition}.
\\ In other words, we obtain $\phi(E_k(M,\xi)) \subset E_k(cM,\phi(\xi))$.
\end{lemma}
\begin{proof}
Let $q_\xi = \frac{k_\xi}{k_{\phi(\xi)}}$, this function belongs to $\Hilbert(\frac{k}{k\circ\phi})$
by (ii), and in \cref{theorem:main} it is noted that
$\|q_\xi\|^2\leq c$. Therefore by Cauchy-Schwarz we have \[
	\left|\frac{k_\xi(x)}{k_{\phi(\xi)}(x)}\right|^2
	= |q_\xi(x)|^2
	\leq \|q_\xi\|^2\frac{k(x,x)}{k\circ\phi(x,x)}
	\leq c\frac{k(x,x)}{k\circ\phi(x,x)}.
\]
\end{proof}

The set inclusion above suggests the following iteration result. Essentially it
says that (some) Denjoy-Wolff points are unique, what about the case $c=1$?
See \cref{question:iteration}.
\begin{theorem}\label{theorem:iteration}
Let $k$ be a reproducing kernel on a nonempty set $X$ that satisfies
\ref{boundary-to-diagonal}-\ref{compact-boundary}.
If $\phi\in\Fact(k)$ satisfies one of
\ref{item:liminf-condition}-\ref{item:uniform-boundary-values} in \cref{theorem:main}
for $\xi\in\rand_kX$ with $c < 1$ and $\phi$ has $E_k$-limit $\xi$ at $\xi$, then the sequence of
iterates $\phi^n(x)\xto{k}\xi$ for every $x\in X$.
\\ In particular $\phi$ has no fixed point in $X$.
\end{theorem}
\begin{proof}
Firstly, since $|k_\xi(x)|>0$ we have $\cup_{M>0}E_k(M,\xi)=X$. So for each $x\in X$
there exists an $M>0$ such that $x\in E_k(M,\xi)$. By \cref{lemma:julia} we have
$\phi^n(x)\in E_k(c^nM,\xi)$, that is \[
	\frac1{c^nM}k(\phi^n(x),\phi^n(x)) \leq |k_\xi(\phi^n(x))|.
\]
Since $k$ satisfies \ref{constants-inclusion}, the left-hand side goes to $\infty$
as $n\to\infty$. Hence $\phi^n(x)\xto{k}\xi$ because of \ref{isolated-singularity}.
\end{proof}

\subsection{Standard weighted Dirichlet spaces on the unit disk}\label{section:dirichlet}
Throughout this section let $k(z,w) = (1-z\overline w)^{-\alpha},\,0<\alpha<1$ be defined
on $\disk$. We shall the denote the corresponding space $\Hilbert(k)$ as $\mathcal{D}_\alpha$.
It consists of analytic functions on $\disk$ such that \[
	\|f\|^2 \coloneq |f(0)|^2 + \int_\disk|f'(z)|^2(1-|z|^2)^\alpha dA(z) < \infty,
\]
where $dA$ is the normalized area measure on $\disk$. As noted in \cref{example:szego-power},
we can identify  $\rand_k\disk = \{\zeta\in\torus: z\mapsto (1-z\overline\zeta)^\alpha\}$ with $\torus$,
and it is apparent that a sequence $(z_n)$ in $\disk$ converges to
$\zeta\in\torus$ if and only if $z_n\xto{k}\zeta$, so $k$-convergence overlaps with
the Euclidean metric.
The family of approach regions for $\zeta\in\torus$ given by $k$ are \begin{align*}
	\Gamma_k(M, \zeta) &= \{z\in\disk:|\zeta-z|  \leq M(1-|z|^2)\} \\
	E_k(M,\zeta)       &= \{z\in\disk:|\zeta-z|^2\leq M(1-|z|^2)\},
\end{align*}
for $M>0$. They are the familiar nontangential and horocyclic approach regions, respectively.

The goal of this section is to prove the following. The condition (i) also appears in section
14 of \cite{hartz_free-outer}.
\begin{corollary}\label{corollary:dirichlet}
Fix $\zeta\in\torus$ and let $\phi$ be a contractive multiplier of $\dirichlet_\alpha$, then
the following are equivalent: \begin{enumerate}[label=(\roman*)]
\item \[
	\liminf_{z\to\zeta}\frac{1-|\phi(z)|^2}{(1-|z|^2)^\alpha} = c < \infty.
\]
\item There is a $\lambda\in\torus$ such that the difference quotient \[
	q_\zeta = \frac{1-\phi(z)\overline\lambda}{(1-z\overline\zeta)^\alpha}
\]
has a nontangential limit $c$ at $\zeta$.
\item The function $z\mapsto \phi'(z)(1-z\overline\zeta)^{1-\alpha}$ has nontangential
limit $c\lambda\overline\zeta\alpha$ at $\zeta$.
\end{enumerate}
Furthermore $c>0$ and $\phi$ has horocyclic limit $\lambda$ at $\zeta$.
\end{corollary}

The equivalence of $(i)$ and $(ii)$ are given by \cref{corollary:row-contractive-multiplier}
with the exception that $c$ is replaced by $\|q_\zeta\|^2$. Their equality
is readily established as $\Gamma_k(M,\zeta)$ approaches $\zeta$ for every $M>\frac12$
and we can take $a_\lambda = 2$ in \ref{boundary-to-diagonal} for any $\lambda\in\torus$
where $t$ is the Szegő kernel. Hence $\|q_\zeta\|^2\leq c \leq (2M\|q_\zeta\|)^2$ and we prove
equality by letting $M\to\frac12^+$. The equivalence of (ii) and (iii) is a matter of analysis
and we formulate it as a lemma.
\begin{lemma}\label{lemma:weighted-derivative}
Let $\zeta\in\torus$ and let $f:\disk\to\plane$ be an analytic function that has nontangential
boundary value $f(\zeta)$ at $\zeta$. Fix a complex number $\sigma$, then the following are
equivalent: \begin{align}
\label{equation:difference-quotient} \frac{f(\zeta)-f(z)}{(1-z\overline\zeta)^\alpha} &\to \sigma, \\
\label{equation:angular-derivative}  f'(z)(1-z\overline\zeta)^{1-\alpha} &\to \sigma\overline\zeta\alpha,
\end{align}
where the limits are taken as $z$ tends to $\zeta$ nontangentially.
\end{lemma}
\begin{proof}
Suppose that \eqref{equation:difference-quotient} holds.
Let $R_M = \{z\in\disk:|\zeta-z|\leq M(1-|z|)\}$ be the typical nontangential approach regions,
note that $\Gamma_k(\frac M2,\zeta)\subset R_M\subset\Gamma_k(M,\zeta)$, hence the respective
convergence notions are equivalent.
For $z\in R_M$ consider the circle $\gamma_z$ centered at $z$ of radius $\frac12(1-|z|)$;
a simple estimate shows that $\gamma_z\subset R_{2M+1}$. Because of
\eqref{equation:difference-quotient} we have \[
	f(z)
	= f(\zeta) - \sigma(1-z\overline\zeta)^\alpha - \beta(z)(1-z\overline\zeta)^\alpha,
\]
where $\beta(z)\to0$ as $z\to\zeta$ nontangentially. By the above and Cauchy's formula
\begin{equation}\label{equation:cauchy-formula}
	f'(z)
	= \sigma\overline\zeta\alpha(1-z\overline\zeta)^{\alpha-1}
		-\frac1{2\pi i}\int_{\gamma_z}\frac{\beta(w)(1-w\overline\zeta)^\alpha}{(z-w)^2}dw
\end{equation}
For $w\in\gamma_z$ we have
$|1-w\overline\zeta| = |\zeta - w| \leq |\zeta - z| + |z - w| \leq (M+\frac12)(1-|z|)$,
therefore we can estimate the integrand as $(2M+1)(1-|z|)^{\alpha-2}\sup_{w\in\gamma_z}|\beta(w)|$.
Multiplying \eqref{equation:cauchy-formula} by $(1-z\overline\zeta)^{1-\alpha}$ and the estimate
in the previous sentence yields \[
	|\sigma\zeta\alpha - f'(z)(1-z\overline\zeta)^{1-\alpha}|
	\leq (2M+1)\sup_{w\in\gamma_z}|\beta(w)|,
\]
which tends to $0$ as $z\to\zeta$ nontangentially.

Conversely suppose that \eqref{equation:angular-derivative} holds.
Let $\gamma(t) = z + t(\zeta - z),\,t\in[0,1]$, by the fundamental theorem of calculus we have \[
	f(\zeta)-f(z)
	= \int_0^1f'(\gamma(t))\gamma'(t))dt
	= \int_0^1f'(\gamma(t))(\zeta-z)dt.
\]
Note of the identity $1-\gamma(t)\overline\zeta = (1-t)(1-z\overline\zeta)$. Therefore,
upon dividing the above equation by $(1-z\overline\zeta)^\alpha$ we obtain \[
	\frac{f(\zeta)-f(z)}{(1-z\overline\zeta)^\alpha}
	= \zeta\int_0^1f'(\gamma(t))(1-z\overline\zeta)^{1-\alpha}\frac{dt}{(1-t)^{1-\alpha}}.
\]
By assumption, as $z\to\zeta$ nontangentially the right-hand side tends to \[
	\sigma\alpha\int_0^1\frac{dt}{(1-t)^{1-\alpha}} = \sigma.
\]
\end{proof}

We end this section with an example showing that there does exist functions $\phi$ satisfying
the conditions of \cref{corollary:dirichlet}.
\begin{example}\label{example:hartz}
For $\zeta\in\torus$ let $\phi(z) = 1 - 1/k_\xi(z) = 1 - (1-z\overline\zeta)^\alpha$,
Proposition 9.2 in \cite{davidson-hartz_ball} shows that $\phi$ is a contractive
multiplier of $\dirichlet_\alpha$. Moreover it is plain to see that \[
	\frac{1-|\phi(z)|^2}{(1-|z|^2)^\alpha}
\]
tends to $1$ as $z\to\zeta$ radially.
\end{example}

\subsection{Standard weighted Besov spaces on the unit ball}\label{section:besov}
Let $\ball_d$ be the unit ball of $\plane^d$ with boundary $\sphere_d$ being the unit sphere.
In this section we consider an analogue for the standard weighted Dirichlet spaces on $\ball_d$.
Let $k(z,w) = (1-\la z,w\ra)^{-\alpha},\,0<\alpha\leq1$ be defined on $\ball_d$.
$\Hilbert(k)$ is realized (with equivalence of norms) as a \emph{standard} weighted
Besov space with the complete Pick property, see \cite{hartz_besov}. We
identify $\rand_k\ball_d = \{\zeta\in\sphere_d:z\mapsto(1-\la z,\zeta\ra)^{-\alpha}\}$
with $\sphere_d$ and note that a sequence $(z_n)$ in $\ball$ converges to $\zeta\in\sphere_d$ if
and only if $z_n\xto{k}\zeta$. Thus we can phrase our results in terms of the Euclidean
metric. The family of approach regions for $\zeta\in\sphere_d$ given by $k$ are
\begin{align*}
	\Gamma_k(M,\zeta) &= \{z\in\sphere_d: |1-\la z,\zeta\ra|  \leq M(1-|z|^2)\} \\
	E_k(M,\zeta)      &= \{z\in\sphere_d: |1-\la z,\zeta\ra|^2\leq M(1-|z|^2)\}
\end{align*}
these are the \emph{Korányi regions} and \emph{horospheres}, respectively.
Notably the Korányi regions allow for \emph{tangential} approach to a
point $\zeta\in\sphere_d$, nonetheless they are often the correct
generalization of nontangential approach in several variables.
Notably horospheres have appeared in the theory of iteration \cite{abate_horospheres}.
Rudin's extension \cite{rudin_unit-ball} of the Julia-Carathéodory theorem to
$\ball_d$ requires a weaker notion than Korányi limits, called \emph{restricted}
Korányi limits, and it is known that this is sharp. If one assumes more of
the functions, then one can extend it to an unrestricted Korányi limit. This is
the main point of the paper \cite{jury_julia-caratheodory}.

By similar considerations as in \cref{section:dirichlet} and by translating \cref{lemma:weighted-derivative}
we obtain the following consequence of \cref{corollary:row-contractive-multiplier}. We note that
the case $\alpha=1$ was proven by Jury, see \cite{jury_julia-caratheodory}.
\begin{corollary}\label{corollary:besov}
Fix $\zeta\in\sphere_d$ and let $\phi$ be a contractive multiplier of $\Hilbert(k)$, then the following
are equivalent: \begin{enumerate}[label=(\roman*)]
\item \[
	\liminf_{z\to\zeta}\frac{1-|\phi(z)|^2}{(1-|z|^2)^\alpha} = c < \infty.
\]
\item There is a $\lambda\in\torus$ such that the difference quotient \[
	q_\zeta(z) = \frac{1-\phi(z)\overline\lambda}{(1-\la z,\zeta\ra)^\alpha}
\]
has a Korányi limit $c$ at $\zeta$.
\item $\phi$ and the function $z\mapsto\nabla\phi(z)(1-\la z,\zeta)^{1-\alpha}$ has
Korányi limits $\lambda\in\torus$ and $c\lambda\overline\zeta\alpha$ at $\zeta$, respectively.
\end{enumerate}
Furthermore $c>0$ and $\phi$ has horospherical limit $\lambda$ at $\zeta$.
\end{corollary}
\begin{remark}
One could also let $\phi$ be a row contractive multiplier. Supposing that
$\phi:\ball_d\to\ball_l$, $l\in \{1,\ldots,\infty\}$. The changes necessary are that
in this case \[
	q_\zeta(z) = \frac{1-\la\phi(z),\lambda\ra}{(1-\la z,\zeta\ra)^\alpha}
\]
in (ii) and the function in (iii) the function
$z\mapsto(\nabla\la\phi(z),\lambda\ra)(1-\la z,\zeta)^{1-\alpha}$
has Korányi limit $c\overline\zeta\alpha$.
\\ Take especially note in the case that $l = d$, then the results in \cref{section:iteration}
applies as well.
\end{remark}

\begin{example}
Just as in \cref{example:hartz}, the function $\phi(z) = 1-1/k_\zeta(z) = 1 - (1-\la z,\zeta\ra)^\alpha$
satisfies the conditions of \cref{corollary:besov}.
\end{example}

\section{Questions}
Throughout this section, if $k$ is not specified otherwise, then it is
assumed to be a reproducing kernel on a nonempty set $X$.

\begin{question}
Let $k(z,w)=(1-z\overline w)^\alpha,\,\alpha>0$, if $\alpha\geq1$, then
as noted in \cref{example:szego-power}, $\Fact(k)$ is equal to the analytic
selfmaps of $\disk$, but this is false for $\alpha<1$. What functions belong to $\Fact(k)$
in this case?
\end{question}

\begin{question}\label{question:riemann-zeta}
Let $k(z,w) = \zeta(z+\overline w),\,z,w\in\plane_\frac12$ as in \cref{example:riemann-zeta}.
As noted there, the reproductive boundary of $k$ equals the critical line
$\{z\in\plane:\Re{z}=\frac12\}$. What are geometric descriptions of the approach $\Gamma_k$
and $E_k$ approach regions? What functions belong to $\Fact(k)$?
\end{question}

\begin{question}[Interpolation with composition factors]
Let $k$ be a reproducing kernel on a nonempty set $X$ and let $Y$ be a subset of $X$.
Suppose that $\phi:Y\to Y$ satisfies \begin{equation}
	\frac{k(x,y)}{k\circ\phi(x,y)} \geq 0, \qquad x,y\in Y.
\end{equation}
Is it possible to find a $\tilde\phi:X\to X$ with $\tilde\phi\restriction_Y=\phi$ such
that $\tilde\phi\in\Fact(k)$?
\end{question}

\begin{question}
In \cref{remark:topology} some properties which are sufficient to guarantee
that the space $X$ is complete under the metric defined by
\eqref{equation:p-metric}, as well as the closed unit balls being compact.
Are these properties necessary?
\end{question}

\begin{question}
Is it possible to extend \cref{corollary:dirichlet} to \emph{any} analytic
selfmap of $\disk$?
\end{question}

\begin{question}[The extreme Denjoy-Wolff points]\label{question:iteration}
When can equality happen in \cref{lemma:julia}? As the proof is via Cauchy-Schwarz,
we know that equality happens if and only if there is a complex number $\sigma$ such
that \[
	q_\xi(y) = \sigma q_x(y)
	\iff \frac{k_\xi(y)}{k_{\phi(\xi)}(\phi(y))} = \sigma\frac{k_x(y)}{k_{\phi(x)}(\phi(y))}
\]
which in turn is equivalent to \begin{equation}
	\frac{k_\xi(y)}{k_x(y)} = \sigma\frac{k_{\phi(\xi)}(\phi(y))}{k_{\phi(x)}(\phi(y))}.
\end{equation}

In the case when $k$ is the Szegő kernel on the
polydisk, are the only functions that satisfy the above equality the analytic automorphisms of $\disk^d$?
\end{question}

\begin{question}[Existence of Denjoy-Wolff points]
Suppose that $k$ satisfies \ref{boundary-to-diagonal}-\ref{compact-boundary}
with non trivial reproductive boundary. If $\phi\in\Fact(k)$ does there
always exist a point $\xi\in\rand_kX$ such that $E_k$-$\lim_{x\to\xi}\phi(x)=\xi$
and (i) with $c\leq1$ i.e. \begin{equation}
	\liminf_{x\to\xi}\frac{k(x,x)}{k\circ\phi(x,x)}=1?
\end{equation}
\end{question}

\section*{\large\centering Acknowledgments}
The author expresses gratitude to the anonymous reviewer whose thoughtful
comments made a significant impact on the quality of the present paper.
Furthermore, he is grateful to his supervisor Alexandru Aleman for guiding him
throughout this project, as well as Adem Limani for many helpful discussions and
comments. Thanks to Michael Hartz for showing \cref{example:hartz}, and Fillipo
Bracci for pointing out the reference in \cref{remark:topology}.

\printbibliography

\end{document}